\documentclass{amsart}

\usepackage{calc} 

\usepackage[T1]{fontenc}
\usepackage{lmodern} 
\usepackage{amsmath,amsthm,amssymb}
\usepackage{enumitem}

\usepackage[
colorlinks=true,
linkcolor=blue
]{hyperref}
\usepackage{algorithm}
\usepackage{algpseudocode}

\usepackage{tikz}

\setlist[enumerate,1]{label=\textnormal{(\alph*)},ref=\textnormal{(\alph*)}}
\setlist[enumerate,2]{label=\textnormal{(\arabic*)},ref=\textnormal{(\arabic*)}}
\newcommand{\lref}[1]{\ref{#1}} 
\newcommand{\tlref}[1]{~\ref{#1}} 
\newcommand{\symblist}[2]{}

\algnewcommand\algorithmicinput{\textbf{Input:}}
\algnewcommand\Input{\item[\algorithmicinput]}
\algnewcommand\algorithmicoutput{\textbf{Output:}}
\algnewcommand\Output{\item[\algorithmicoutput]}
\algnewcommand\algorithmicbreak{\textbf{break}}
\algnewcommand\Break{\algorithmicbreak}
\algnewcommand\algorithmiccontinue{\textbf{continue}}
\algnewcommand\Continue{\algorithmiccontinue}
\algnewcommand\algorithmicset{\textbf{set}}
\algnewcommand\Set[1]{#1}
\algnewcommand\algorithmicgoto{\textbf{goto}}
\algnewcommand\Goto[0]{\algorithmicgoto{} }

\newcommand{\dprobfont}[1]{\textsc{#1}}
\newcommand{\dprobfontinp}[1]{\textsc{#1}}

\newlength{\lenl}
\newlength{\lenm}
\newlength{\lenr}

\newcommand{\vproblem}[5]{
\setlength{\lenm}{\maxof{\widthof{#2\quad}}{\widthof{#4\quad}}}
\setlength{\lenr}{\textwidth-\lenl-\lenm}
\begin{flalign*}\arraycolsep=0pt
\begin{tabular}{@{}p{\lenl}@{}p{\lenm}@{}p{\lenr}@{}}
& \multicolumn{2}{@{}l@{}}{#1} \\
& \begin{minipage}[t]{\lenm}#2\vspace{1.5pt}\end{minipage} 
& \begin{minipage}[t]{\lenr}#3\vspace{1.5pt}\end{minipage} \\ 
& \begin{minipage}[t]{\lenm}#4\end{minipage} 
& \begin{minipage}[t]{\lenr}#5\end{minipage} \\
\end{tabular}&&
\end{flalign*}
}
\newcommand{\dproblem}[3]{\vproblem{\dprobfontinp{#1}}{Input:}{#2}{Problem:}{#3}}
\newcommand{\cproblem}[3]{\vproblem{\dprobfontinp{#1}}{Input:}{#2}{Output:}{#3}}

\def\clap#1{\hbox to 0pt{\hss#1\hss}}

\def\mathrlap{\mathpalette\mathrlapinternal}
\def\mathclap{\mathpalette\mathclapinternal}

\def\mathrlapinternal#1#2{%
\rlap{$\mathsurround=0pt#1{#2}$}}
\def\mathclapinternal#1#2{%
\clap{$\mathsurround=0pt#1{#2}$}}

\newcommand{\underbr}[2]{\underbrace{#1}_{\mathclap{#2}}}

\newcommand{\ptime}{\textnormal{P}}
\newcommand{\np}{\textnormal{NP}}
\newcommand{\pspace}{\textnormal{PSPACE}}

\newcommand{\exptime}{\textnormal{EXPTIME}}

\newcommand{\smp}{\dprobfont{SMP}}
\newcommand{\sat}{\dprobfont{SAT}}

\newcommand{\gap}{\textnormal{GAP}}
\newcommand{\true}{\text{true}}
\newcommand{\false}{\text{false}}

\newcommand{\N}{\mathbb{N}}

\newcommand{\alg}[1]{#1}
\newcommand{\var}{\mathcal}
\newcommand{\subuni}[1]{ {\langle #1 \rangle} }

\newcommand{\sst}{\mid}
\newcommand{\rg}[1]{[#1]}
\newcommand{\gj}{\mathrel{\mathcal{J}}}
\newcommand{\gh}{\mathrel{\mathcal{H}}}
\newcommand{\gr}{\mathrel{\mathcal{R}}}
\newcommand{\gl}{\mathrel{\mathcal{L}}}
\newcommand{\gd}{\mathrel{\mathcal{D}}}

\newcommand{\gjle}{\leq_{\mathcal{J}}}
\newcommand{\gjge}{\geq_{\mathcal{J}}}
\newcommand{\gjless}{<_{\mathcal{J}}}
\newcommand{\gjgt}{>_{\mathcal{J}}}

\newcommand{\grle}{\leq_{\mathcal{R}}}

\newcommand{\glle}{\leq_{\mathcal{L}}}

\newcommand{\grc}{\mathrel{\mathcal{R}}}
\newcommand{\glc}{\mathrel{\mathcal{L}}}
\newcommand{\gjc}{\mathrel{\mathcal{J}}}
\newcommand{\gjlec}{\le_{\mathcal{J}}}
\newcommand{\gjgec}{\ge_{\mathcal{J}}}

\newcommand{\Tsg}{T}

\newcommand{\qidand}{\mathrel{\&}} 
\newcommand{\qidarr}{\longrightarrow} 
\newcommand{\dwrd}{\bar}
\newcommand{\dexp}{\overline}
\newcommand{\dsgp}[1]{\bar{#1}}

\newtheorem{thm}{Theorem}[section]

\newtheorem{lma}[thm]{Lemma}
\newtheorem{clr}[thm]{Corollary}

\theoremstyle{definition}

\newtheorem{dfn}[thm]{Definition}

\newtheorem{prb}[thm]{Problem}

\theoremstyle{remark}

\newcommand{\lam}{\textit{$\lambda$}}
\newcommand{\lamd}{\textit{$\dwrd\lambda$}}
\newcommand{\cpinfix}{\dprobfont{Infix}}
\newcommand{\cpsuffix}{\dprobfont{Suffix}}

\newcommand{\qd}{d}
\newcommand{\qe}{e}

\newcommand{\qh}{h}
\newcommand{\qx}{x}
\newcommand{\qy}{y}
\newcommand{\acs}{\hspace{\arraycolsep}}

\begin{document}
\title[The SMP for bands]{The subpower membership problem for bands}
\author{Markus Steindl}
\address{Institute for Algebra, Johannes Kepler University Linz, Altenberger St 69, 4040 Linz, Austria}
\address{Department of Mathematics, University of Colorado Boulder, Campus Box 395, Boulder, Colorado 80309-0395}
\email{\href{mailto:markus.steindl@colorado.edu}{markus.steindl@colorado.edu}}
\thanks{Supported by the Austrian Science Fund (FWF): P24285}
\date{\today}
\begin{abstract}
Fix a finite semigroup $S$ and let $a_1,\ldots,a_k, b$ be tuples in a direct power $S^n$.
The subpower membership problem (\smp) for $S$ asks whether $b$ can be generated by $a_1,\ldots,a_k$. For bands (idempotent semigroups), we provide a dichotomy result:
if a band $S$ belongs to a certain quasivariety, then $\smp(S)$ is in \ptime;
otherwise it is \np-complete.

Furthermore we determine the greatest variety of bands all of whose finite members induce a tractable $\smp$. Finally we present the first example of two finite algebras that generate the same variety and have tractable and \np-complete \smp{}s, respectively.
\end{abstract}
\maketitle

\section{Introduction}
\label{sec:1}

How hard is deciding membership in a subalgebra of a given algebraic structure?
This problem occurs frequently in symbolic computation.
For instance if $\alg F$ is a fixed field and we are given
vectors $a_1,\ldots,a_{k},b$ in a vector space $\alg F^n$,
we often want to decide whether $b$ is in the linear span of $a_1,\ldots,a_k$.
This question can be answered using Gaussian elimination in 
polynomial time in $n$ and $k$.

Depending on the formulation of the membership problem, 
the underlying algebra may be part of the input.
For example if we are given transformations on $n$ elements,
we may have to decide whether they generate a given transformation under composition.
These functions belong to the full transformation semigroup $\Tsg_n$ on $n$ elements.
In this case $n$ and the algebra $\Tsg_n$ are part of the input.
Kozen proved that this problem is \pspace-complete \cite{Kozen1977b}.
However, if we restrict the input to permutations on $n$ elements, then the problem is in \ptime{}
using Sims' stabilizer chains \cite{FHL1980}.

In this paper we investigate the membership problem formulated by Willard in 2007 \cite{Willard2007}. 
Fix a finite algebra $S$ with finitely many basic operations. 
We call a subalgebra of some direct power of $S$ a \emph{subpower} of $S$.
The \emph{subpower membership problem} $\smp(S)$ is the following decision problem:
\dproblem{SMP(\textit{S})}{
$\{a_1,\ldots,a_k\} \subseteq S^n,b \in S^n$}{
Is $b$ in the subalgebra of $S^n$ generated by $\{a_1,\ldots,a_k\}$?}
In this problem the algebra $S$ is not part of the input.

The \smp{} is of particular interest within the study of the constraint satisfaction problem (CSP) \cite{IMMVW2010}.
Recall that in a CSP instance the goal is to assign values of a given domain to a set of variables such that each constraint is satisfied. Constraints are usually represented by constraint relations.
In the algebraic approach to the CSP, each relation is regarded as a subpower of a certain finite algebra $S$.
Instead of storing all elements of a constraint relation, we can store a set of generators. Checking whether a given tuple belongs to a constraint relation represented by its generators is precisely the 
\smp{} for $S$.

The input size of $\smp(S)$ is essentially $(k+1)n$.
We can always decide the problem using a straightforward closure algorithm in time exponential in $n$.
For some algebras there is no faster algorithm. This follows from a result of Kozik \cite{Kozik2008}, who actually constructed a finite algebra with \exptime-complete \smp{}.
However, there are structures whose \smp{} is considerably easier.
For example, the \smp{} for a finite group is in \ptime{} by an adaptation of Sims' stabilizer chains~\cite{Zweckinger2013}.
Mayr \cite{Mayr2012} proved that the \smp{} for Mal'cev algebras is in~\np.
He also showed that the \smp{} for every finite Mal'cev algebra which has prime power size and a nilpotent reduct is in \ptime{}.

In this paper we investigate the \smp{} for \emph{bands} (idempotent semigroups).
For semigroups in general the \smp{} is in \pspace{} by a result of Bulatov, Mayr, and the present author \cite{BMS2015}.
There is no better upper bound since various semigroups were shown to have a \pspace-complete \smp{},
including the full transformation semigroup on $3$ or more letters and
the $6$-element Brandt monoid \cite{BMS2015,smppspace}.
For commutative semigroups, however, the \smp{} is in \np.
In~\cite{BMS2015} a dichotomy result was provided:
if a commutative semigroup $S$ embeds into a direct product of a Clifford semigroup and a nilpotent
semigroup, then $\smp(S)$ is in P; otherwise it is \np-complete~\cite{BMS2015}.
These were also the first algebras known to have an \np-complete \smp.
Similar to the case of commutative semigroups,
we will establish a \ptime/\np-complete dichotomy for the \smp{} for bands in the present paper.
Before that we introduce the notions of variety, identity, and quasiidentity.

Let $v,w$ be words over variables $x_1,\ldots,x_k$.
For a semigroup $S$ we define the function
\symblist{wS}{$w^S$}$w^S \colon S^k \to S$ such that,
when applied to $(s_1,\ldots,s_k)\in S^k$, 
it replaces every occurrence of $x_i$ in $w$ by $s_i$ for all $i\in\{1,\ldots,k\}$ and computes the resulting product.
We refer to $w^S$ as the ($k$-ary) \emph{term function}\index{term function} induced by $w$.
An expression of the form \symblist{v=w}{$v \approx w$}%
$v \approx w$ is called an \emph{identity}\index{identity!equational law} over $x_1,\ldots,x_k$.
A semigroup $S$ 
\emph{satisfies} the identity\index{satisfaction!of an identity}
$v \approx w$ (in symbols \symblist{=}{$\models$}%
$S \models v\approx w$) if 
\[ \forall s_1,\ldots,s_k \in S \colon v^S(s_1,\ldots,s_k)=w^S(s_1,\ldots,s_k). \]
A class $\var V$ of semigroups is called a \emph{variety}\index{variety of semigroups} if
there is a set $\Sigma$ of identities such that $\var V$ contains precisely the semigroups that satisfy all identities in $\Sigma$.

Given identities $v_1\approx w_1,\ldots,v_m \approx w_m$ and $p\approx q$ over $x_1,\ldots,x_k$, we call
an expression $\mu$ of the form
\begin{equation*}\label{eq_example1_qid} 
v_1\approx w_1 \qidand\ldots\qidand v_m \approx w_m \qidarr p\approx q 
\end{equation*}
a \emph{quasiidentity}\index{quasiidentity} over $x_1,\ldots,x_k$. A semigroup $S$ 
\emph{satisfies}\index{satisfaction!of a quasiidentity} the quasiidentity 
$\mu$ (in symbols \symblist{=}{$\models$}$S \models \mu$) if for all $\bar s = (s_1,\ldots,s_k)$ in $S^k$,
\[ v_1^S(\bar s)=w_1^S(\bar s),\ldots,v_m^S(\bar s)=w_m^S(\bar s) \quad\text{implies}\quad p^S(\bar s)=q^S(\bar s). \]
A class $\var V$ of semigroups is a \emph{quasivariety}\index{quasivariety of semigroups} if
there is a set $\Sigma$ of quasiidentities such that $\var V$ contains precisely the semigroups that satisfy all identities in $\Sigma$.

Green's equivalences are denoted by $\gl,\gr,\gj,\gd,\gh$ \cite[p.~45]{Howie1995}. For the preorders $\glle,\grle,\gjle$ see \cite[p.~47]{Howie1995}.

The dichotomy result for the \smp{} for bands is based upon the following two quasiidentities.
\hypertarget{ht:lam}{}
\begin{equation}\label{eq:q1p}\tag{\lam}
\arraycolsep=0pt
\left(\begin{array}{c}
\begin{array}{rl}
\qd \qx \qy \qe {}&{}\approx \qd \qe \\
\qh \qx {}&{}\approx \qx \\
\qh \qe {}&{}\approx \qe \end{array}\\
\qd \gjle \qe \gjle \qx, \qy
\end{array}\right) \qidarr \qd \qx \qe \approx \qd \qe.
\end{equation}
\hypertarget{ht:lamd}{}
\begin{equation}
\arraycolsep=0pt
\label{eq:q1pd}
\tag{\lamd}
\left(\begin{array}{c}\begin{array}{rl}
\qe \qy \qx \qd &{}\approx \qe \qd \\
\qx \qh &{}\approx \qx \\
\qe \qh &{}\approx \qe 
\end{array} \\
\qd \gjle \qe \gjle \qx, \qy
\end{array}\right)\qidarr \qe \qx \qd \approx \qe \qd.
\end{equation}

\noindent
We will see that for every band the last condition on the left hand side of both \lam{} and \lamd{} is equivalent to
\[ \qd \qe \qd \approx \qd \qidand \qe \qx \qe \approx \qe  \qidand  \qe \qy \qe \approx \qe. \]
This will follow from Lemma~\ref{lma:band_rule1}.
Thus the atomic formulas of \lam{} and \lamd{} are really identities.
Also note that \lamd{} is the quasiidentity obtained when each word of \lam{} is reversed. 
We obtain the following observation.

\begin{lma}\label{lma:dualsgp_dualqid}
A semigroup $S$ satisfies $\dsgp \lambda$ if and only if the dual semigroup $\dsgp S$ satisfies $\lambda$.
\end{lma}
\begin{proof}Straightforward.\end{proof}

We state the main result of this paper.

\begin{thm}\label{thm:dichotomy_bands_intro}
If a finite band $S$ satisfies \hyperlink{ht:lam}{\lam} and \hyperlink{ht:lamd}{\lamd}, then $\smp(S)$ is in \ptime. 
Otherwise $\smp(S)$ is \np-complete.
\end{thm}
\noindent
For the proof see Section~\ref{sec:np_completeness_qvs}.
This means that under the assumption $\ptime\ne\np$, the finite bands with tractable \smp{} are precisely the finite members of the quasivariety determined by \lam{} and \lamd.
This quasivariety is not a variety by the following theorem.

\begin{thm}\label{thm:genvar_diffcomp}
There are a 9-element band $S_9$ and a 10-element band $S_{10}$ such that
\begin{enumerate}
\item\label{it1_thm:genvar_diffcomp} $S_9$ and $S_{10}$ generate the same variety,
\item\label{it2_thm:genvar_diffcomp} $S_9$ is a homomorphic image of $S_{10}$, and
\item\label{it3_thm:genvar_diffcomp} $\smp(S_{10})$ is in \ptime, whereas $\smp(S_9)$ is \np-complete.
\end{enumerate}
\end{thm}

\noindent
See Definition~\ref{dfn:s9s10} for the multiplication tables of $S_9$ and $S_{10}$.
This is the first example of two finite algebras which generate the same variety
and induce tractable and \np-complete \smp{}s, respectively.
The band $S_{10}$ is also the first finite algebra known to have a tractable \smp{} and a homomorphic image with \np-hard \smp.
We will prove Theorem~\ref{thm:genvar_diffcomp} in Section~\ref{sec:genvar_diffcomp}.

In the characterization of all varieties of bands \cite{GP1989}, the sequences of words $G_n$, $H_n$, and $I_n$ over $x_1,\ldots,x_n$ for $n\ge2$ play a fundamental role. 
We list the first four words of each sequence.

\begin{dfn}[cf. {\cite[Notation 5.1]{GP1989}}]
\label{dfn:bands_Gn_Hn_In}
\begin{align*}\begin{array}{llll}
n & G_n				& H_n 								 & I_n                                  \\
2 & x_2 x_1		& x_2								 & x_2x_1x_2                            \\
3 & x_3x_1x_2		& x_3x_1x_2x_3x_2					 & x_3x_1x_2x_3x_2x_1x_2					 \\
4 & x_4x_2x_1x_3	& x_4x_2x_1x_3x_4x_2x_3x_2x_1x_3 & x_4x_2x_1x_3x_4x_2x_1x_2x_3x_2x_1x_3	 \\
\end{array}\end{align*}
\end{dfn}
For two words $v$ and $w$ we let $[v\approx w]$ denote the variety of bands that satisfy $v \approx w$.
We call a variety \emph{proper} if it is smaller than the variety of all bands.
By $\dwrd{v}$ we denote the \emph{dual word} of a word $v$, which is the word obtained when 
the order of the variables of $v$ is reversed. 

\begin{thm}[{\cite[Diagram 1]{GP1989}}]
Every variety of bands is of the form {$[v\approx w]$} for some identity $v\approx w$.
The lattice of proper varieties of bands is given by Figure~\ref{fig:varieties_of_bands}.
\end{thm}

\newcommand{\varbrkts}[1]{{[#1]}}
\newcommand{\varapp}{\approx}
\begin{figure}[p]
\centering
\resizebox{!}{1.0\textheight-4\baselineskip}{
\begin{tikzpicture}[x=1cm, y=1cm]
  \node (c0) at (0,-2.5) {$\varbrkts{x \varapp y}$};
  \node (b1) at (-2,-0.5) {$\varbrkts{xy \varapp x}$};
  \node (d1) at (2,-0.5) {$\varbrkts{xy \varapp y}$};
  \node (c2) at (0,1.5)  {$\varbrkts{xyx \varapp x}$};
  \draw[thick] (c0)--(b1)--(c2)--(d1)--(c0);

  \node (c1) at (0,0)   {$\varbrkts{       xy \varapp yx       }$};

  \node (b2) at (-2,2)  {$\varbrkts{      zxy \varapp zyx      }$};
  \node (d2) at (2,2)   {$\varbrkts{      xyz \varapp yxz      }$};
  \draw[thick] (b2)--(c1)--(d2);

  \node (a3) at (-4,4)  {$\varbrkts{ G_2 \varapp I_2 }$}; 
  \node (c3) at (0,4)   {$\varbrkts{zxyz \varapp zyxz}$};
  \node (e3) at (4,4)   {$\varbrkts{\dwrd G_2 \varapp \dwrd I_2 }$};
  
  \draw[thick] (a3)--(b2)--(c3)--(d2)--(e3);
  \draw[thick] (c0)--(c1) (b1)--(b2) (d1)--(d2) (c2)--(c3);
  
  \node (b4) at (-2,6)  {$\varbrkts{\dwrd G_3G_3 \varapp \dwrd I_3H_3 }$};
  \node (d4) at (2,6)   {$\varbrkts{\dwrd G_3G_3 \varapp \dwrd H_3I_3 }$};
  \draw[thick] (a3)--(b4)--(c3)--(d4)--(e3);
  
  \node (a5) at (-4,8)  {$\varbrkts{ G_3 \varapp H_3 }$};
  \node (c5) at (0,8)   {$\varbrkts{\dwrd G_3G_3 \varapp \dwrd I_3I_3 }$};
  \node (e5) at (4,8)   {$\varbrkts{\dwrd G_3 \varapp \dwrd H_3 }$};
  \draw[thick] (a5)--(b4)--(c5)--(d4)--(e5);

  \node (b6) at (-2,10)  {$\varbrkts{\dwrd G_4G_3 \varapp \dwrd H_4I_3 }$};
  \node (d6) at (2,10)   {$\varbrkts{\dwrd G_3G_4 \varapp \dwrd I_3H_4 }$};
  \draw[thick] (a5)--(b6)--(c5)--(d6)--(e5);
  
  \node (a7) at (-4,12)  {$\varbrkts{ G_3 \varapp I_3 }$};
  \node (c7) at (0,12)   {$\varbrkts{\dwrd G_4G_4 \varapp \dwrd H_4H_4 }$};
  \node (e7) at (4,12)   {$\varbrkts{\dwrd G_3 \varapp \dwrd I_3 }$};
  \draw[thick] (a7)--(b6)--(c7)--(d6)--(e7);

  \node (b8) at (-2,14)  {$\varbrkts{\dwrd G_4G_4 \varapp \dwrd I_4H_4 }$};
  \node (d8) at (2,14)   {$\varbrkts{\dwrd G_4G_4 \varapp \dwrd H_4I_4 }$};
  \draw[thick] (a7)--(b8)--(c7)--(d8)--(e7);
  
  \node (a9) at (-4,16)  {$\varbrkts{ G_4 \varapp H_4 }$};
  \node (c9) at (0,16)   {};
  \node (e9) at (4,16)   {$\varbrkts{\dwrd G_4 \varapp \dwrd H_4 }$};
  \draw[thick] (a9)--(b8)--(c9)--(d8)--(e9);

  \node (x1) at (-3,17)  {};
  \node (x2) at (-1,17)  {};
  \node (x3) at (1,17)   {};
  \node (x4) at (3,17)   {};
  \draw[thick] (a9)--(x1) (c9)--(x2) (c9)--(x3) (e9)--(x4);
\end{tikzpicture}
}
\caption{%
The lattice of proper varieties of bands, taken from \cite{GP1989}.
For two words $v$ and $w$ the expression $\varbrkts{v\varapp w}$ denotes the variety of bands that satisfy the identity $v\varapp w$.}
\label{fig:varieties_of_bands}
\end{figure}
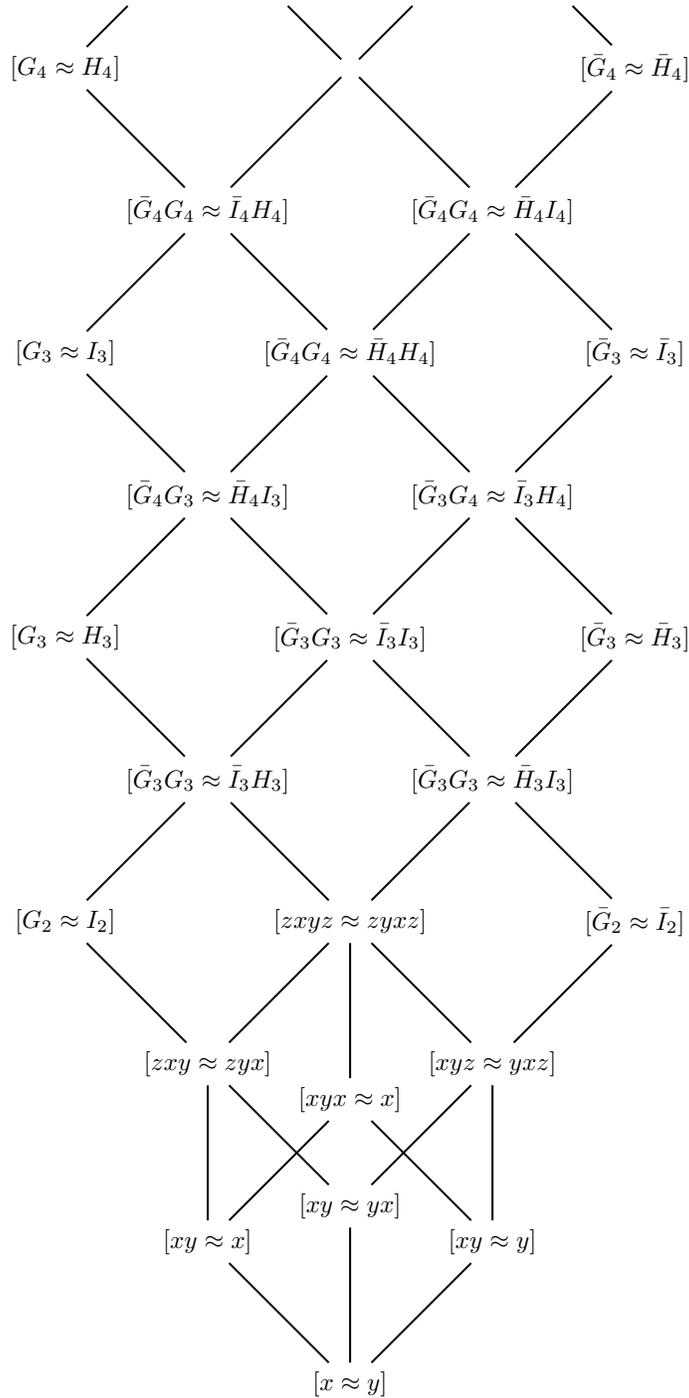

The following result on the \smp{} for bands will be proved in Section~\ref{sec:genvar_diffcomp}.

\begin{thm}\label{thm:greatest_var_smp_tract}
Assume \/$\ptime\neq\np$. Then $[\dwrd G_4 G_4\approx \dwrd H_4 H_4]$ is the greatest variety of bands all of whose finite members induce a tractable \smp{}.
\end{thm}

Bands that satisfy the identity $\dwrd G_3 G_3\approx \dwrd I_3 I_3$ are called \emph{regular}.
From Theorem~\ref{thm:greatest_var_smp_tract} and Figure \ref{fig:varieties_of_bands} we obtain the following result.
\begin{clr}
The \smp{} for every regular band is in \ptime{}.
\end{clr}

\section{Varieties of bands}

The following lemma states the well-known fact that every band is a \emph{semilattice of rectangular bands}. A band is called \emph{rectangular} if it satisfies $xyz \approx xz$.
\begin{lma}[cf. {\cite[Theorem 4.4.1]{Howie1995}}] 
\label{lma:band_j_cong}
Let $S$ be a band.
\begin{enumerate}
\item\label{it1_lma:band_j_cong}
$\gj$ is a congruence on $S$.
\item\label{it2_lma:band_j_cong}
$S/{\gj}$ is a semilattice, and for all $x,y \in S$ we have $xy \gj x$ if and only if $x \gjle y$.
\item\label{it3_lma:band_j_cong}
Each $\gj$-class is a rectangular band.
\end{enumerate}
\end{lma}

\begin{lma}[{\cite[Lemma~2.2]{GP1989}}]
\label{lma:band_rule0}
Let $x,y,z$ be elements of a band $S$ such that $x \gj z$. Then 
$x \gjle y$ if and only if $x y z = xz$.
\end{lma}

\begin{proof}
First assume $x\gjle y$.
Lemma~\ref{lma:band_j_cong}\tlref{it2_lma:band_j_cong} implies $x \gj xy$. We have $xyz = x(xy)z = xz$ by idempotence and Lemma~\ref{lma:band_j_cong}\tlref{it3_lma:band_j_cong}.

For the converse assume $xyz=xz$. 
Since $S/{\gj}$ is a semilattice, we have $x \gj xz = xyz$. Thus $xy \gj x$, and Lemma~\ref{lma:band_j_cong}\tlref{it2_lma:band_j_cong} yields $x \gjle y$.
\end{proof}

We will use the following well-known rules throughout this paper.

\begin{lma}
\label{lma:band_rule1}
Let $x,y$ be elements of a band $S$. Then 
\begin{enumerate}
	\item\label{it1_lma:band_rule1} $x \gjle y$\quad{}if and only if\quad$x y x = x$,
	\item\label{it2_lma:band_rule1} $x \glle y$\quad{}if and only if\quad$x y = x$,
	\item\label{it3_lma:band_rule1} $x \grle y$\quad{}if and only if\quad$y x = x$.
\end{enumerate}
\end{lma}

\begin{proof}
\lref{it1_lma:band_rule1} is immediate from Lemma~\ref{lma:band_rule0}.

\lref{it2_lma:band_rule1}
By the definition of $\glle$ and by idempotence we have
\[ x \glle y \quad\text{iff}\quad \exists u \in S^1 \colon uy=x \quad\text{iff}\quad xy=x. \]

\lref{it3_lma:band_rule1} is proved similarly to \lref{it2_lma:band_rule1}.
\end{proof}

From Lemma~\ref{lma:band_rule1} follows that 
$x\gjle y$ holds in $S$ if and only if it holds in 
some subsemigroup of $S$ that contains $x$ and $y$.
The same is true for $\glle$ and $\grle$.

We write $\rg{n}:=\{1,\ldots,n\}$ for $n\in\N$.
A tuple $a$ in a direct power $S^n$ is considered as a function $a\colon \rg{n} \rightarrow S$. Thus the $i$th coordinate of this tuple is denoted by $a(i)$ rather than $a_i$.
The subsemigroup generated by a set $A = \{ a_1,\ldots,a_k \}$ may be denoted by $\subuni{A}$ or $\subuni{a_1,\ldots,a_k}$.

\begin{lma}\label{lma:gjle_band_componentwise}
Let $S$ be a band and $x,y \in S^n$ for some $n\in\N$. Let $\le$ be one of the preorders $\gjle,\glle,\grle$.
Then 
\begin{equation*}
x \le y \quad\text{if and only if}\quad \forall i\in\rg{n} \colon x(i) \le y(i).
\end{equation*}
\end{lma}

\begin{proof}
We prove the statement for $\gjle$.
By Lemma~\ref{lma:band_rule1}\tlref{it1_lma:band_rule1} the following are equivalent:
\begin{align*}\begin{aligned} 
&x \gjle y, \\
&xyx=x, \\
&xyx(i)=x(i) &&\text{for all } i\in\rg{n}, \\
&x(i) \gjle y(i) &&\text{for all } i\in\rg{n}.
\end{aligned}\end{align*}
For  $\glle$ and  $\grle$ the equivalence is proved similarly.
\end{proof}

In the remainder of this section, we use some well-established results on the lattice of varieties of bands. For a full characterization the reader is referred to~\cite{GP1989}. The following notation was introduced there.

\begin{dfn}[cf. {\cite[Notation 2.1]{GP1989}}] 
\mbox{} \\
\def\arraystretch{1.25}
\arraycolsep=0pt
\begin{tabular}{@{}p{\widthof{$F(X)$\quad}}@{}p{\textwidth-\widthof{\quad$F(X)$\quad}}}
$\dsgp{S}$ & the \emph{dual semigroup} of $( S, \cdot )$ is the semigroup $( S,* )$ with $x*y := y \cdot x$. \\
$S^1$ & the semigroup obtained when an identity is adjoined to $S$. \\
$X$ & the countably infinite set of variables $\{x_1,x_2,\ldots\}$. \\
$F(X)$ & the free semigroup over $X$. \\
$\varnothing$ & the empty word, i.e.\ the identity of $F(X)^1$. \\
$\dwrd{w}$ & 
the \emph{dual} of a word $w \in F(X)^1$, i.e.\ the word obtained when reversing the order of the variables of $w$. \\
$c(w)$ & the \emph{content} of a word $w \in F(X)^1$, i.e.\ the set of variables occurring in $w$. \\
$s(w)$ &
the longest left cut of the word $w$ that contains all but one of the variables of $w$:
For $w \neq \varnothing$ let $u,v \in F(X)^1$ and $x \in X$ such that $w = uxv$ and $c(u)\ne c(ux) = c(w)$. 
Then define $s(w) := u$. Let $s(\varnothing) := \varnothing$. \\
$\sigma(w)$ &
for $w \neq \varnothing$, the last variable in $w$ under the order of the first occurrence, starting from the left.
We define $\sigma( \varnothing ) := \varnothing$. \\
$\dwrd{f}$ &
for a (partial) function $f \colon F(X)^1 \to F(X)^1$, we define $\dwrd{f}(w) := \dexp{ f (\dwrd{w}) }$.
\end{tabular}
\end{dfn}

\begin{dfn}[cf. {\cite[Notation 3.1]{GP1989}}]\label{dfn:varbands_hn} 
For $n \geq 2$ we define $h_n \colon F(X)^1 \rightarrow F(X)^1$,
\begin{align*}
h_n(\varnothing) &:= \varnothing  \quad\text{for } n \geq 2,\\
h_2(w) &:=\text{the first variable of $w$ if $w\ne\varnothing$,} \\
h_n(w) &:= h_n s(w)\sigma(w)\dwrd h_{n-1}(w)  \quad\text{for } n \geq 3,\ w\ne\varnothing.
\end{align*}
\end{dfn}

\noindent
From~\cite{GP1989} we obtain the following upper bound on the length of $h_n(w)$ for $n\ge2$. This will allow us to prove that the \smp{} for every band is in \np.

\begin{lma}\label{lma:smp_band_in_np}
For every integer $n\ge 2$ there is a polynomial $p_n$ such that for all $k \in \mathbb N$ and all $k$-ary terms $t$ the length of $h_n(t)$ is at most $p_n(k)$.
\end{lma}
\begin{proof}
We use induction on $n$. For the base case note that the length of $h_2(t)$ is $1$ for all terms $t$. 
Now assume the assertion is true for some $n\ge2$.
Let \[ p_{n+1}(k) := k(1+p_n(k))\quad\text{for } k\in\mathbb{N}. \]
Let $k\in\mathbb{N}$, and $t$ be a term over $x_1,\ldots,x_k$. Let $\ell\le k$ be the number of variables actually occurring in $t$. 
By repeated application of the recursion in Definition~\ref{dfn:varbands_hn} we obtain
\begin{align*}
\begin{aligned} 
h_{n+1}(t) &= h_{n+1} s(t)\cdot \sigma(t)\dwrd h_n(t) \\
&= h_{n+1} s^2(t)\cdot \sigma s(t)\dwrd h_ns(t)\cdot \sigma(t) \dwrd h_n(t) \quad \text{} \\
&\mathrel{\makebox[\widthof{=}]{\vdots}} \\
&= h_{n+1} s^\ell(t)\cdot \sigma s^{\ell-1}(t) \dwrd h_ns^{\ell-1}(t)\,\cdots\,\sigma s^1(t) \dwrd h_ns^1(t)\cdot \sigma(t) \dwrd h_n(t) \\
&= \prod_{i=\ell-1}^0 \sigma s^{i}(t) \dwrd h_ns^{i}(t)\qquad\text{since $s^\ell(t) = \varnothing$.} \\
\end{aligned}
\end{align*}
The length of $\sigma s^{i}(t)$ is 1 for all $i$ by the definition of $\sigma$. Thus the length of each factor of the product is at most $1+p_n(k)$. Therefore the length of $h_{n+1}(t)$ is bounded by $\ell (1+p_n(k))$, which is at most $p_{n+1}(k)$.
\end{proof}

\begin{lma}\label{lma:smp_band_in_np_2}
Let $S$ be a finite band.
Then there is a polynomial $p$ such that every $k$-ary term function on $S$ is induced by a term of length at most $p(k)$.
\end{lma}
\begin{proof}
It is well-known that the variety of bands is not finitely generated.
Thus $S$ belongs to a variety $[G_n \approx H_n]$ for some $n \geq 3$ by \cite{GP1989}. 
Let $p := p_n$ be the polynomial from Lemma~\ref{lma:smp_band_in_np}.
Now let $f$ be a $k$-ary term function on $S$ induced by some term $t$.
By \cite[Theorem 4.5]{GP1988} $S$ satisfies $t \approx h_n(t)$.
Thus $h_n(t)$ also induces $f$.
By Lemma~\ref{lma:smp_band_in_np} the length of $h_n(t)$ is at most $p_n(k)$. 
\end{proof}

\begin{thm}\label{thm:smp_band_in_np}
The \smp{} for a finite band $S$ is in \np.
\end{thm}

\begin{proof}
Fix an instance $\{a_1,\ldots, a_k\} \subseteq S^n,\, b \in S^n$ of $\smp(S)$. If $b \in \subuni{a_1,\ldots,a_k}$, then there is a term function $f$ such that $b = f( a_1,\ldots, a_k )$.
By Lemma~\ref{lma:smp_band_in_np_2} $f$ is induced by a term $t$ whose length is at most $p(k)$.
We can verify $b = f(a_1,\ldots, a_k )$ in $\mathcal{O}(np(k))$ time.
Thus $t$ is a witness for $b \in \subuni{a_1,\ldots,a_k}$.
\end{proof}

\section{Quasiidentities}
\label{sec:quid}

Recall the quasiidentities \hyperlink{ht:lam}{\lam} and \hyperlink{ht:lam}{\lamd} from Section \ref{sec:1}. For every finite band $S$ we introduce two intermediate problems, $\cpinfix(S)$ and $\cpsuffix(S)$. If $S$ satisfies \lam, these problems can be solved in polynomial time.
If $S$ also satisfies the dual quasiidentity \lamd, then $\smp(S)$ is in \ptime.
In Section~\ref{sec:np_completeness_qvs} we will show that the \smp{} for the remaining finite bands is \np-complete.

For finite bands $S$, we define $\cpinfix(S)$ as follows.
\cproblem{\cpinfix(\textit{S})}{
$c,d,e \in S^n$ such that $c \gjc d$ and $d\gjlec e$, \\
$A \subseteq S^n$ such that $\forall a \in A \colon e \gjlec a$.}
{Some $y \in \subuni A$ such that $dye = c$ if it exists; \false{} otherwise.}
We call an output $y \in \subuni A$ a \emph{solution} of $\cpinfix(S)$.
The next result allows us to combine and reduce solutions of $\cpinfix(S)$ if $S$ satisfies the quasiidentity \lam.
\begin{lma}\label{lma:cp_infix_1}
Let $S$ be a finite band that satisfies \lam{}.
Let $c,d,e\in S^n,\,A\subseteq S^n$ be an instance of $\cpinfix(S)$ and $x,y,z \in\subuni{A}$.
\begin{enumerate}
\item\label{it1_lma:cp_infix}
If $xy$ and $z$ are solutions of $\cpinfix(S)$ and there exists $h \in S^n$ with $hx = x$ and $hz = z$, then $xz$ is a solution of $\cpinfix(S)$.
\item\label{it2_lma:cp_infix}
If $xyz$ is a solution of $\cpinfix(S)$ with $x \gjlec y$ and $x \gjlec z$, 
then $xz$ is also a solution of $\cpinfix(S)$.
\end{enumerate}
\end{lma}

\begin{proof}
\lref{it1_lma:cp_infix}
Let $x,y,z,h$ be as as above. We claim that
\begin{equation}\label{eq1_lma:cp_infix}
\arraycolsep=0pt
\begin{array}{c}\begin{array}{rl}
dx(ye)(ze)&{}=d(ze) \\
hx&{}=x \\
h(ze)&{}=(ze) 
\end{array} \\
d \gjlec (ze) \gjlec x,(ye).
\end{array}
\end{equation}
For proving \eqref{eq1_lma:cp_infix} note that the condition on $A$ and Lemma~\ref{lma:band_j_cong} imply $e \gjlec g$ for all $g \in \subuni A$. 
So $e \gjlec z$, and thus $eze = e$ by Lemma~\ref{lma:band_rule1}\tlref{it1_lma:band_rule1}.
Hence $dxyeze=dxye=c=dze$, proving the first statement. 
The next two statements are clear from the assumptions.
Lemma~\ref{lma:band_j_cong} and $y,z \in \subuni A$ imply $e \gjc ye \gjc ze$. 
This and $d \gjlec e \gjlec x$ prove the last statement.
By \lam{} the identities \eqref{eq1_lma:cp_infix} imply $dx(ze)=d(ze)$. Thus $dxze=c$. Item~\lref{it1_lma:cp_infix} is proved.

\lref{it2_lma:cp_infix}
Let $x,y,z$ be as above. We prove that
\begin{equation}\label{eq2_lma:cp_infix}
\arraycolsep=0pt
\begin{array}{c}\begin{array}{rl}
(dxy)(zx)y(ze)&{}=(dxy)(ze) \\
z(zx)&{}=(zx) \\
z(ze)&{}=(ze)
\end{array} \\
d \gjlec (ze) \gjlec (zx),y.
\end{array}
\end{equation}
The first statement holds since $x(yz)x=x$ by Lemma~\ref{lma:band_rule1}\tlref{it1_lma:band_rule1}.
Statements two and three are clear.
Lemma~\ref{lma:band_j_cong} implies $e \gjc ze$ and $x \gjc zx$.
This and $d \gjlec e \gjlec x,y$ prove the last statement.
Now \eqref{eq2_lma:cp_infix} implies $(dxy)(zx)(ze)=(dxy)(ze)$ by \lam.
As $xyzx=x$, we obtain $dxze=dxyze$.
This proves item~\lref{it2_lma:cp_infix}.
\end{proof}

\begin{thm}\label{thm:cp_infix_in_p}
Let $S$ be a finite band that satisfies \lam.
Then Algorithm~\ref{alg:cp_infix} solves $\cpinfix(S)$ in polynomial time.
\end{thm}

\begin{algorithm}
\caption{\newline
Solves $\cpinfix(S)$ in polynomial time if the band $S$ satisfies \lam.}
\label{alg:cp_infix}
\begin{algorithmic}[1]
\Input $c,d,e \in S^n,\,A\subseteq S^n$ as in the definition of $\cpinfix(S)$.
\Output{$y \in \subuni A$ such that $dye = c$ if it exists; \false{} otherwise.}
	\For{$a_0 \in A$ with $\exists s \in S^n \colon s \gjgec e,\, d a_0 s e = c$ }
	\label{alg:cp_infix_for1}
	\label{alg:cp_infix_find_s}
			\State\Set{$s := a_0 s$}	\label{alg:cp_infix_set_s_1}
			\State\Set{$y := a_0$}		\label{alg:cp_infix_set_y_1}
			\algblock{until}{e}
				\until\ $\exists a_1 \in A \colon a_1 \gjgec y,\, dya_1e = c$ \textbf{do}
				\label{alg:cp_infix_if1}
				\If{$\exists a_2,a_3 \in A \colon a_2 \gjgec y,\, a_3 \not \gjgec y,\, dya_2a_3se = c 
					$}		\label{alg:cp_infix_if2}
					\State\Set{$y := y a_2 a_3$}	\label{alg:cp_infix_set_y_2}
				\Else	\label{alg:cp_infix_else}
					\State\Continue{} for loop
				\EndIf
				\label{alg:cp_infix_endif}
			\e \textbf{nd until}
			\State\Return $ya_1$
			\Comment{$a_1$ from line~\ref{alg:cp_infix_if1}}
			\label{alg:cp_infix_ret_true}
	\EndFor
	\State \Return{\false}
	\label{alg:cp_infix_ret_false}
\end{algorithmic}
\end{algorithm}

\begin{proof}
Before we deal with correctness and complexity, we need some preparation.
We fix an instance $c,d,e \in S^n,\,A \subseteq S^n$ of $\cpinfix(S)$.
We claim that
\begin{align}
\label{cl1_thm:cp_infix_in_p}
\text{the until loop iterates at most $\mathcal O(n)$ times.}
\end{align}
Let $h$ be the height of the semilattice $S/{\gj}$.
It suffices to show that the if branch in line~\ref{alg:cp_infix_if2}
is entered at most $n(h-1)$ times.
Whenever this happens, $y$ is modified.
This modification has the form $y := y a_2 a_3$.
By line~\ref{alg:cp_infix_if2} $a_3 \not \gjgec y$, and thus $y a_2 a_3 \not \gjgec y$. We have
\begin{align*}\begin{split}
&\forall i \in \rg{n} \colon y a_2 a_3(i) \gjle y(i), \\
&\exists i \in \rg{n} \colon y a_2 a_3(i) \gjless y(i).
\end{split}\end{align*}
Hence $ya_2a_3$ is strictly smaller than $y$ in the preorder $\gjlec$. We can decrease $y$ at most $n(h-1)$ times. 
We proved~\eqref{cl1_thm:cp_infix_in_p}. In particular the algorithm stops.

\emph{Correctness of Algorithm~\ref{alg:cp_infix}}.
First assume Algorithm~\ref{alg:cp_infix} returns some $z\ne{}$\false.
This can only happen in line~\ref{alg:cp_infix_ret_true}, and thus $z = y a_1$.
From line~\ref{alg:cp_infix_if1} follows $d y a_1 e = c$.
We have $y \in \subuni A$ since
the only lines where $y$ is modified are~\ref{alg:cp_infix_set_y_1} and~\ref{alg:cp_infix_set_y_2}.
So $z = y a_1$ is a solution of $\cpinfix(S)$.

Conversely assume that $\cpinfix(S)$ has a solution $z \in \subuni A$.
Our goal is to show that Algorithm~\ref{alg:cp_infix} returns some solution rather than \false.
Let $b_1,\ldots,b_m \in A$ such that $z=b_1 \cdots b_m$.
For the remainder of the proof of correctness we may assume that 
\begin{equation}\label{eq10_alg:cp_infix}
\text{the variable $a_0$ of the for loop is set to $b_1$.}
\end{equation}
It suffices to show that the algorithm returns some solution in line~\ref{alg:cp_infix_ret_true} for this case.
The $s$ in line~\ref{alg:cp_infix_find_s} exists since $s := z$ is one possibility. Fix the value assigned to $s$ in line~\ref{alg:cp_infix_set_s_1}.
For each value assigned to $y$ in Algorithm~\ref{alg:cp_infix}, we claim that
\begin{align}\label{cl2_thm:cp_infix_in_p}
y \in \subuni A, \quad a_0 y = y ,\quad \text{and} \quad dyse = c.
\end{align}
If $y$ obtained its value in line~\ref{alg:cp_infix_set_y_1}, then clearly~\eqref{cl2_thm:cp_infix_in_p} holds.
If not, then $y$ obtained its value by one or more calls of the assignment $y := y a_2 a_3$ in line~\ref{alg:cp_infix_set_y_2}. So the first two statements in~\eqref{cl2_thm:cp_infix_in_p} follow by induction, and line \ref{alg:cp_infix_if2} implies $dyse = c$.

Next we claim that for each value assigned to $y$,
\begin{align}\label{cl3_thm:cp_infix_in_p}
\text{if $y \gjlec z$, then the condition in line~\ref{alg:cp_infix_if1} is fulfilled.}
\end{align}
Fix such $y \gjlec z$. By~\eqref{cl2_thm:cp_infix_in_p} and the assumptions we have
\[ dyse = c, \quad dze = c, \quad a_0 y = y, \quad a_0 z = z. \]
Now apply Lemma~\ref{lma:cp_infix_1}\tlref{it1_lma:cp_infix} and obtain $dyze = c$.
Note that instead of $A$ we use $A':=\{ a \in S^n \sst a \gjge e \}$ for the hypothesis of Lemma~\ref{lma:cp_infix_1}.
For $a_1:=b_m$ we have $z=za_1$ and thus
\begin{align*}
dyza_1e = c,   \quad y \gjlec z \gjlec a_1.
\end{align*}
Lemma~\ref{lma:cp_infix_1}\tlref{it2_lma:cp_infix} yields $dya_1e = c$.
This proves $\eqref{cl3_thm:cp_infix_in_p}$.

Similar to~\eqref{cl3_thm:cp_infix_in_p} we claim that for each value assigned to $y$,
\begin{align}\label{cl4_thm:cp_infix_in_p}
\text{if $y \not \gjlec z$, then the condition in line~\ref{alg:cp_infix_if2} is fulfilled.}
\end{align}
Fix such $y \not \gjlec z$. 
Let $b_i$ be the first generator of $z$ with $b_i \not \gjgec y$. 
By \eqref{eq10_alg:cp_infix} and \eqref{cl2_thm:cp_infix_in_p} we have $b_1=a_0 \gjgec y$. Thus $b_i\neq b_1$. 
Let $z_1 := b_1\cdots b_{i-1}$, $a_2 := b_{i-1}$, and $a_3:=b_i$.
Idempotence implies $z_1 a_2 a_3 z = z$.
We have
\begin{align*}\begin{aligned}
d(z_1a_2a_3)ze &= c,	&& \\
dse &= c 				&&\text{by lines~\ref{alg:cp_infix_find_s} and~\ref{alg:cp_infix_set_s_1},}\\
a_0 z_1 &= z_1		&&\text{by~\eqref{eq10_alg:cp_infix},} \\
a_0 s &= s 			&&\text{by line~\ref{alg:cp_infix_set_s_1}.}
\end{aligned}\end{align*}
We apply Lemma~\ref{lma:cp_infix_1}\tlref{it1_lma:cp_infix} and obtain $d (z_1 a_2 a_3) s e = c$. Now we have
\begin{align*}\begin{aligned}
dyse &= c && \text{by~\eqref{cl2_thm:cp_infix_in_p},} \\
d (z_1 a_2 a_3 s) e &= c, && \text{} \\
a_0 y &= y && \text{by~\eqref{cl2_thm:cp_infix_in_p},} \\
a_0 z_1 &= z_1.		&&
\end{aligned}\end{align*}
Apply Lemma~\ref{lma:cp_infix_1}\tlref{it1_lma:cp_infix} again and obtain $d y (z_1 a_2 a_3 s) e = c$.
For $e':=a_3se$ we have $e' \gj e$ and
\begin{align*}\begin{aligned}
d y z_1 a_2 e' = c,& &&\text{} \\
e' \gjlec y \gjlec z_1 \gjlec a_2& &&\text{by the definitions of $z_1$ and $a_2$.}
\end{aligned}\end{align*}
From Lemma~\ref{lma:cp_infix_1}\tlref{it2_lma:cp_infix} we obtain $d y a_2 e' = d y a_2 (a_3 s e) = c$.
This proves~\eqref{cl4_thm:cp_infix_in_p}.

We are ready to complete the proof of correctness.
By the assumption of~\eqref{eq10_alg:cp_infix} and
by~\eqref{cl3_thm:cp_infix_in_p} and~\eqref{cl4_thm:cp_infix_in_p}, 
Algorithm~\ref{alg:cp_infix} does not enter the else branch in line~\ref{alg:cp_infix_else}.
By~\eqref{cl1_thm:cp_infix_in_p} the until loop stops after $\mathcal{O}(n)$ iterations.
After that $y a_1$ is returned in line~\ref{alg:cp_infix_ret_true}.
By~\eqref{cl2_thm:cp_infix_in_p} and line~\ref{alg:cp_infix_if1}, $y a_1$ is a solution of $\cpinfix(S)$.

\emph{Complexity of Algorithm~\ref{alg:cp_infix}}.
By~\eqref{cl1_thm:cp_infix_in_p} the until loop iterates at most $\mathcal{O}(n)$ times.
Evaluating the condition in line~\ref{alg:cp_infix_if1} requires $\mathcal{O}(|A|)$ multiplications in $S^n$, and the condition in line~\ref{alg:cp_infix_if2} requires $\mathcal{O}(|A|^2)$ multiplications.
Thus $\mathcal{O}(n|A|^2)$ multiplications in $S$ are performed in one iteration of the until loop.
Therefore the effort for the until loop is $\mathcal{O}(n^2|A|^2)$.

The tuple $s$ in line~\ref{alg:cp_infix_find_s} can be found componentwise.
In particular, for each $i \in \rg{n}$ we find $s(i) \in S$ such that $da_0se(i) = c(i)$.
This process requires $\mathcal{O}(n)$ steps.
Lines~\ref{alg:cp_infix_set_s_1} and~\ref{alg:cp_infix_set_y_1} require $\mathcal{O}(n)$ steps.
Altogether one iteration of the for loop requires $\mathcal{O}(n^2|A|^2)$ time.
Hence Algorithm~\ref{alg:cp_infix} runs in $\mathcal{O}(n^2|A|^3)$ time.
\end{proof}

For finite bands $S$ we define the intermediate problem $\cpsuffix(S)$.
We will show that $\cpsuffix(S)$ is in \ptime{} if $S$ satisfies \hyperlink{ht:lam}{\lam}.
\cproblem{\cpsuffix(\textit{S})}{
$A \subseteq S^n ,\, b \in S^n$.}{
Some $x \in \subuni A$ such that $x \glc b$ if it exists; \false{} otherwise.}
As usual we call an output $x \in \subuni A$ a \emph{solution} of $\cpsuffix(S)$.
Every solution $x$ fulfills $bx = b$, and thus $x$ is a `suffix' of $b$.
Hence the name of the problem.

\begin{thm}
\label{thm:cp_suffix_in_p}
Let $S$ be a finite band that satisfies \lam.
Then Algorithm~\ref{alg:psuffix} solves $\cpsuffix(S)$ in polynomial time.
\end{thm}

\begin{algorithm}
\caption{\newline{}%
Solves $\cpsuffix(S)$ in polynomial time if the band $S$ satisfies \lam.}
\label{alg:psuffix}
\begin{algorithmic}[1]
\Input $A \subseteq S^n ,\, b \in S^n$
\Output{$x \in \subuni A$ such that $x \glc b$ if it exists; \false{} otherwise.}
	\State find $x \in A$ such that $b x = b$
	\label{alg:psuffix_find_x}
	\State \Return \false{} if no such $x$ exists
	\label{alg:psuffix_ret_false_1}
	\While{$x \not \glc b$ }
	\label{alg:psuffix_while}
			\State find $a \in A,\,a \gjge b,\,a \not \gjgec x$ and $
				y \in \subuni A,\, y \gjge x$ such that $(b a) y x = b$
			\label{alg:psuffix_find_y}
			\item[]
			\Comment{at most $|A|$ instances of $\cpinfix(S)$}
		\label{alg:psuffix_endfor}
			\State\Return \false{} if no such $a,y$ exist
			\State\Set{$x := ayx$}		\label{alg:psuffix_set_x}
	\EndWhile
	\label{alg:psuffix_ret_while_end}
	\State\Return $x$
	\label{alg:psuffix_ret_true}
\end{algorithmic}
\end{algorithm}

\begin{proof}
Before we prove correctness and complexity, we need some preparation.
To show that the algorithm always stops, we claim that
\begin{align}\label{cl1_thm:cp_suffix_in_p}
\text{the while loop iterates at most $\mathcal O(n)$ times.}
\end{align}
Let $h$ be the height of the semilattice $S/{\gj}$.
In each iteration, either the algorithm terminates, or $x$ is modified by $x := a y x$.
By line~\ref{alg:psuffix_find_y} $a \not \gjgec x$.
Thus also $a y x \not \gjgec x$.
We have
\begin{align*}\begin{split}
&\forall i \in \rg{n} \colon a y x(i) \gjle x(i), \\
& \exists i \in \rg{n} \colon a y x(i) \gjless x(i).
\end{split}\end{align*}
Therefore the number of modifications of $x$ is at most $n(h-1)$.
We proved~\eqref{cl1_thm:cp_suffix_in_p}.

Later in the proof we use the following:
\begin{align}\label{eq15_thm:cp_suffix_in_p}
x \glc  b   \quad\text{if and only if}\quad   x \gjc b   \text{ and }   bx = b.
\end{align}
This follows immediately since every $\gj$-class is a rectangular band.

\emph{Correctness of Algorithm~\ref{alg:psuffix}}.
First we claim that
\begin{align}\begin{split}\label{eq10_thm:cp_suffix_in_p}
& \text{each value assigned to $x$ fulfills $x \in \subuni A$ and $bx = b$.}
\end{split}\end{align}
After initializing $x$ in line~\ref{alg:psuffix_find_x},~\eqref{eq10_thm:cp_suffix_in_p} clearly holds.
The value of $x$ is possibly modified by $x := ayx$ in line~\ref{alg:psuffix_set_x}.
If $x \in \subuni A$, then also $ayx \in \subuni A$.
From line~\ref{alg:psuffix_find_y} we know that $bayx = b$. 
This proves~\eqref{eq10_thm:cp_suffix_in_p}.

Now assume Algorithm~\ref{alg:psuffix} returns some $x \neq\false$.
Then the while loop has finished. This implies $x \glc b$.
By~\eqref{eq10_thm:cp_suffix_in_p} $x$ is a solution of $\cpsuffix(S)$.

Conversely assume $\cpsuffix(S)$ has a solution $z \in \subuni A$.
Our goal is to show that Algorithm~\ref{alg:psuffix} returns some solution.
We fix $c_1,\ldots,c_m \in A$ such that $z=c_1 \cdots c_m$.
The $x$ in line~\ref{alg:psuffix_find_x} exists.
For instance for $x=c_m$ we have $b = bz = bzx = bx$.

Now fix a value assigned to $x$ in line~\ref{alg:psuffix_find_x} or line~\ref{alg:psuffix_set_x}.
After this assignment, the while loop is called.
If $x \glc b$, then the algorithm 
returns $x$, which is a solution by~\eqref{eq10_thm:cp_suffix_in_p}.
Assume $x \not \glc b$. 
By~\eqref{eq15_thm:cp_suffix_in_p} and~\eqref{eq10_thm:cp_suffix_in_p}
we have $x \gjgt b$. We claim that $a$ and $y$ as defined in line~\ref{alg:psuffix_find_y} exist,
that is
\begin{equation}
\label{eq_20_thm:cp_suffix_in_p}
\exists a \in A,\,a \gjge b,\,a \not \gjgec x \ 
\exists y \in \subuni A,\, y \gjge x \colon b a y x = b.
\end{equation}
Since $z \gjc b$ and $b \not \gjgec x$, there is a $c_i$ with $c_i \not \gjgec x$.
Assume $i\in\rg{m}$ is maximal with this property.
If $i=m$, then let $y := x$.
Otherwise let $y := c_{i+1}\cdots c_{m}$.
In both cases $y \in \subuni A$ and $y \gjge x$.
Idempotence implies $zx = zc_iyx$.
Hence
\begin{align*}
b = bzx = bzc_iyx = bc_iyx.
\end{align*}
Note that $c_i \gjgec b$ since $S/{\gj}$ is a semilattice.
This proves~\eqref{eq_20_thm:cp_suffix_in_p}.
Now we know that \false{} is never returned in the while loop.
By~\eqref{cl1_thm:cp_suffix_in_p} the while loop finishes after finitely many iterations.
After that $x \glc b$ holds, and the solution $x$ is returned.

\emph{Complexity of Algorithm~\ref{alg:psuffix}}.
Line~\ref{alg:psuffix_find_x} requires at most $\mathcal{O}(n|A|)$ multiplications in $S$.
By~\eqref{cl1_thm:cp_suffix_in_p} the while loop iterates at most $\mathcal O(n)$ times.
In line~\ref{alg:psuffix_find_y} the algorithm iterates over $a$ and $y$. 
If we consider $a$ as fixed, then
the algorithm tries to find $y \in \subuni A,\,y \gjge x$ such that $(ba)yx=b$. 
Finding such $y$ is an instance of $\cpinfix(S)$.
Thus in line~\ref{alg:psuffix_find_y} at most $|A|$ instances of $\cpinfix(S)$ have to be solved.
By the proof of Theorem~\ref{thm:cp_infix_in_p}, one instance runs in time $\mathcal{O}(n^2|A|^3)$.
Thus the while loop requires $\mathcal{O}(n^3|A|^4)$ steps.
Altogether Algorithm~\ref{alg:psuffix} runs in time $\mathcal{O}(n^3|A|^4)$.
\end{proof}

\begin{thm}
\label{thm:smp_in_p}
Let $S$ be a finite band that satisfies \lam{} and \lamd.
Then Algorithm~\ref{alg:smp} decides $\smp(S)$ in polynomial time.
\end{thm}

\begin{algorithm}
\caption{\newline{}%
Decides $\smp(S)$ in polynomial time if the band $S$ satisfies \lam{} and \lamd.}
\label{alg:smp}
\begin{algorithmic}[1]
\Input $A \subseteq S^n,\, b \in S^n$
\Output{\true{} if $b \in \subuni{ A }$; \false{} otherwise. }
	\State\Return $\exists x,y \in \subuni{A} \colon b \glc x \wedge b \grc y$
	\label{alg:smp_ret}
\end{algorithmic}
\end{algorithm}

\begin{proof}
\emph{Correctness of Algorithm~\ref{alg:smp}}.
If Algorithm~\ref{alg:smp} returns \true, then 
$bx=b=yb$. Thus $b=ybx=yx$ by Lemma~\ref{lma:band_rule1}\tlref{it1_lma:band_rule1} and hence $b\in\subuni A$.
Conversely assume $b \in \subuni A$.
Then $x$ and $y$ as defined in line~\ref{alg:smp_ret} exist since we can set $x = y = b$.
Algorithm~\ref{alg:smp} returns \true.

\emph{Complexity of Algorithm~\ref{alg:smp}}.
In line~\ref{alg:smp_ret} one instance of $\cpsuffix(S)$ and one of $\cpsuffix(\dsgp S)$ are solved.
By the proof of Theorem~\ref{thm:cp_suffix_in_p} and by Lemma~\ref{lma:dualsgp_dualqid} both can be decided in $\mathcal{O}(n^3|A|^4)$ time.
\end{proof}

\section{\np-hardness}
\label{sec:np_completeness_qvs}

In the previous section we showed that the \smp{} for a finite band that satisfies both
\hyperlink{ht:lam}{\lam} and \hyperlink{ht:lam}{\lamd} is in \ptime. In this section we will prove our dichotomy result Theorem~\ref{thm:dichotomy_bands_intro} by showing that the \smp{} for the remaining finite bands is \np-complete.

\begin{dfn}
\label{dfn:witness_q1p}
Let $S$ be a band.
We say $d,e,x,y,h$ \emph{witness} $S\not\models\lam$ if they satisfy the premise of \lam, but not the implication.
\end{dfn}

\noindent
Witnesses have the following properties.

\begin{lma}
\label{lma:witness1}
Let $S$ be a finite band such that $d,e,x,y,h\in S$ witness $S\not\models\lam$. Then
\begin{enumerate}
\item\label{lma:witness1_it01} $d \gjless e \gjless x \gjless h$,
\item\label{lma:witness1_it03} $e, xe, y$ are distinct,
\item\label{lma:witness1_it02} $d, dx, de, dxe$ are distinct.
\end{enumerate}
\end{lma}

\begin{proof}
\lref{lma:witness1_it01}
We have
\begin{align*}\begin{aligned}
\label{eq10:lma:witness1}
&\qd \qx \qy \qe = \qd \qe, \quad \qh \qx = \qx, \quad \qh \qe = \qe ,\\
&\qd \gjle \qe \gjle \qx, \qy \quad\text{and}\quad \qd \qx \qe \neq \qd \qe.
\end{aligned}\end{align*}
Thus $d \gjle e \gjle x \gjle h$.
We show the strictness of each inequality by assuming the opposite and deriving the contradiction $dxe = de$. \\
If $d \gj e$, then $dxe = de$ by Lemma~\ref{lma:band_rule0}. \\
If $e \gj x$, then $xe = xye$ by Lemma~\ref{lma:band_rule0}.
Thus $dxe = dxye = de$. \\
If $x \gj h$, then $hxh = h$ by Lemma~\ref{lma:band_rule1}\tlref{it1_lma:band_rule1}. Thus $dxe = d(hx)(he) = dhe = de$. 

\lref{lma:witness1_it03}
Assuming any equality, we derive the contradiction $dxe = de$. \\
If $e = xe$, then $dxe = de$. \\
If $e = y$, then $dxe = dxee = dxye = de$. \\
If $xe = y$, then $dxe = dxxee = dxye = de$.

\lref{lma:witness1_it02}
Assuming any equality, we derive the contradiction $dxe = de$. \\
If $d = dx$, then $dxe = de$. \\
If $d = de$, then $dxe = dexe = de$ since $exe = e$. \\
If $d = dxe$, then $dxe = dxee = de$. \\
If $dx = de$, then $dxe = dee = de$. \\
If $dx = dxe$, then $dxe = dxeye = dxye$ since $e = eye$. \\
\end{proof}

If $S \not \models \lam$, then there are witnesses with additional properties by the following lemma.

\renewcommand{\qh}{1}
\newcommand{\aw}[1]{\bar{#1}}

\begin{lma}\label{lma:T}
Let $S$ be a finite band that does not satisfy \lam.
Then there are $d,e,x,y,h \in S$ such that for $T := \subuni{ d,e,x,y,h }$ the following holds:
\begin{enumerate}
\item\label{it1_lma:T}
$h$ is the identity of $T$, and we have the following partial multiplication table whose entries are distinct.
\begin{equation*}\label{eq0_lma:T}
\begin{array}{c|ccccc}
 T &  x &  e & xe &  y &  d  \\
\hline
 x &  x & xe & xe &  y & xd  \\
 e &  e &  e &  e &  e &  d  \\
xe & xe & xe & xe & xe & xd  \\
 y &  y &  y &  y &  y & yd  \\ 
 d & dx & de &dxe & de &  d \end{array}
\end{equation*}
\item\label{it2_lma:T} 
$d,e,x,y,h$ witness $T\not\models\lam$ and $S\not\models\lam$.
\item\label{it3_lma:T}
$d/{\gj}$ is the smallest $\gj$-class of $T$, \\
$d/{\gr} = \{d, dx, de, dxe\}$, where the $4$ elements are distinct, and \\
$d/{\gl} = \{d, xd, yd\}$.
\item\label{it4_lma:T}
Either of the following holds:
\begin{enumerate}
\item\label{it4d_lma:T} $d = xd = yd$ and $|T|=9$.
\item\label{it4b_lma:T} $d = xd \neq yd $ and $|T|=13$.
\item\label{it4c_lma:T} $d \neq xd = yd $ and $|T|=13$.
\item\label{it4a_lma:T} $d, xd, yd$ are distinct and $|T|=17$.
\end{enumerate}
\end{enumerate}
\end{lma}

\begin{proof} 
Let $\aw d, \aw e, \aw x, \aw y, \aw h \in S$ witness $S\not\models\lam$. Define
\begin{align*}\begin{aligned}
d &:= \aw e\aw x\aw h\aw d \aw h, & e &:= \aw e\aw x\aw h, & x &:= \aw x \aw h,   \\
y &:= \aw x\aw y\aw e\aw x\aw h, & h &:= \aw h. && 
\end{aligned}\end{align*}

\lref{it1_lma:T}
Since $\aw h$ is a left identity for $\aw x$ and $\aw e$ and by idempotence, $h=\aw h$ is an identity for $d,e,x,y$.
In the first row of the multiplication table, 
the only nontrivial entry is 
\newcommand{\reqvspace}{
\phantom{(\aw e\aw x\aw h\underbr{\aw d \aw h) (\aw x\aw y\aw e %
     }{=\aw d\aw x\aw y\aw e=\aw d\aw e=\aw d\aw h\aw e}
 \aw x\aw h) 
     = (\aw e\aw x\aw h\aw d \aw h) (\aw e\aw x\aw h) = de.}}
\[ 
 xy = \mathrlap{
   \underbr{(\aw x\aw h)(\aw x}{\aw x}\aw y\aw e\aw x\aw h) 
 = \aw x\aw y\aw e\aw x\aw h = y.
}\reqvspace
\]
For the second row we use idempotence and Lemma~\ref{lma:band_rule1}\tlref{it1_lma:band_rule1}. We obtain
\begin{align*}
 ex &= \mathrlap{
 (\aw e\underbr{\aw x\aw h)(\aw x\aw h)}{\aw x\aw h} = \aw e\aw x\aw h = e,
}\reqvspace \\
 ey &= (\aw e\underbr{\aw x\aw h)(\aw x\aw y}{\gjge\aw e}\aw e\aw x\aw h) = \aw e\aw x\aw h = e.
\end{align*}
The remaining entries follow from these.
The third row is immediate from the second one.
For the last two rows it suffices to show that
\begin{align*}
 yx &= (\aw x\aw y\aw e\underbr{\aw x\aw h)(\aw x \aw h)}{\aw x\aw h} = \aw x\aw y\aw e\aw x\aw h = y, \\
 ye &= (\aw x\aw y\underbr{\aw e\aw x\aw h) (\aw e\aw x\aw h)}{\aw e\aw x\aw h} 
     = \aw x\aw y\aw e\aw x\aw h = y, \\
 dy &= (\aw e\aw x\aw h
 	\underbr{\aw d \aw h) (\aw x\aw y\aw e }
 	{=\aw d\aw x\aw y\aw e=\aw d\aw e=\aw d\aw h\aw e}
 	\aw x\aw h) 
     = (\aw e\aw x\aw h\aw d \aw h) (\aw e\aw x\aw h) = de.
\end{align*}
We will show that $x,e,xe,y,d$ are distinct after proving \lref{it2_lma:T}.

\lref{it2_lma:T}
First we show that
\begin{align}\label{eq30_lma:T}
dxye = de \neq dxe.
\end{align}
The equality follows from the multiplication table in~\lref{it1_lma:T}.
Now suppose $dxe = de$. Then $\aw d dxe \aw e = \aw d de \aw e$.
By Lemma~\ref{lma:band_rule1}\tlref{it1_lma:band_rule1} we obtain
\begin{align*}
\aw d dxe \aw e 
&= \underbr{\aw d(\aw e\aw x\aw h\aw d}{\aw d}\underbr{\aw h)(\aw x\aw h)}{\aw x}\underbr{(\aw e\aw x\aw h)\aw e}{\aw e} = \aw d\aw x\aw e, \\
\aw d de \aw e 
&= \underbr{\aw d(\aw e\aw x\aw h\aw d}{\aw d}\underbr{\aw h)(\aw e\aw x\aw h)\aw e}{\aw e} = \aw d\aw e.
\end{align*}
Now $\aw d\aw x\aw e = \aw d\aw e$, which is impossible.
Hence \eqref{eq30_lma:T} holds. 

From the multiplication table we see that
\begin{align}\label{eq40_lma:T}
d \gjle e \gjle x,y.
\end{align}
Now~\eqref{eq30_lma:T}, \eqref{eq40_lma:T}, 
$hx=x$, and $he=e$ yield item~\lref{it2_lma:T}.

For item~\lref{it1_lma:T} it remains to show that $x,e,xe,y,d$ are distinct.
By item~\lref{it2_lma:T} and Lemma~\ref{lma:witness1}\tlref{lma:witness1_it03} $e,xe,y$ are distinct.
From Lemma~\ref{lma:witness1}\tlref{lma:witness1_it01} follows $d \gjless e \gjless x \gjless h$, and from the multiplication table $e \gj xe \gj y$.
Thus the elements are all distinct.

\lref{it3_lma:T}
By~\eqref{eq40_lma:T} and since $T/{\gj}$ is a semilattice by Lemma~\ref{lma:band_j_cong}, $d/{\gj}$ is the smallest $\gj$-class of $T$. Thus we have $d/{\gr}=dT$ and $d/{\gl}=Td$ in $T$.
The multiplication table yields 
\[ dT=d\, \subuni{d,e,x,y,h}=\{d,dx,de,dxe\}.\]
The elements are distinct by Lemma~\ref{lma:witness1}\tlref{lma:witness1_it02}.
For $Td$ we obtain
\begin{align*}
&&\begin{aligned}
Td
 &= \subuni{d,e,x,y,h} d &&\\
 &= \subuni{e,x,y,h} d  
&&\text{since $dud = d$ for $u \in \subuni{e,x,y,h}$,} \\
 &= \subuni{x,y,h} d  
&&\text{since $eud = eued = ed = d$ for $u \in \subuni{x,y,h}$,} \\
&= \{xd,yd,d\}  &&\text{by the multiplication table.}\\
\end{aligned}
\end{align*}
Item~\lref{it3_lma:T} is proved.

\lref{it4_lma:T} We count the elements of $T$. Note that $T = \{h,x,e,xe,y\} \cup d/{\gj}$ and
\[ d/{\gj} = \{\ell dr \sst \ell \in \{h,x,y\},\, r \in \{h,x,e,xe\}\}. \] 
If $d=yd$, then $xd=xyd=yd=d$, and item~\lref{it4_lma:T}\lref{it4d_lma:T} holds.
If $d\ne yd$, then one of the remaining cases applies.
Using \gap{} \cite{GAP4} and the semigroups package \cite{GAP_semigroups},
it is easy to check that the semigroups for the cases
\lref{it4_lma:T}\lref{it4d_lma:T} to \lref{it4_lma:T}\lref{it4a_lma:T}
actually exist.
\end{proof}

\begin{clr}
\label{clr:witness2}
Let $S$ be a finite band. Then $S$ does not satisfy \lam{} if and only if one of the four non-isomorphic bands $T$ described in Lemma~\ref{lma:T} embeds into $S$. 
\end{clr}

\begin{proof}
First we show that the two 13-element bands are not isomorphic. 
Let $T := \subuni{d,e,x,y,h}$ and $\aw T := \subuni{\aw d,\aw e,\aw x,\aw y,\aw h}$ be bands that fulfill properties \lref{it1_lma:T} to~\lref{it3_lma:T} of Lemma~\ref{lma:T}. Assume $T$  fulfills~\lref{it4_lma:T}\lref{it4b_lma:T}, and $\aw T$ fulfills~\lref{it4_lma:T}\lref{it4c_lma:T}. Suppose there is an isomorphism $\alpha \colon T \rightarrow \aw T$. Isomorphisms preserve the relations $\gjle$ and $\gj$. We have $T/{\gj}=\{\{h\},\{x\},\{e,xe,y\},d/{\gj}\}$ and a similar partition for $\aw T$.
Thus $\alpha$ maps $h$ to $\aw h$ and $x$ to $\aw x$. We apply $\alpha$ to the inequality $xe\neq e$ and obtain $\aw x \alpha(e) \neq \alpha(e)$. 
In order to fulfill the latter inequality and as $\alpha(e)\in\{\aw e,\aw x\aw e,\aw y\}$, we have $\alpha(e) = \aw e$. Therefore $\alpha(xe) = \aw x \aw e$, and thus $\alpha(y) = \aw y$.
By items~\lref{it3_lma:T} and~\lref{it4_lma:T}\lref{it4c_lma:T} of Lemma~\ref{lma:T} there is an $\aw \ell \in \{\aw h, \aw x\}$ such that
\begin{equation}\label{eq1_clr:witness2}
\alpha(d) \in \{\aw\ell\aw d, \aw\ell\aw d\aw x, \aw\ell\aw d\aw e, \aw\ell\aw d\aw x\aw e\}.
\end{equation}
Lemma~\ref{lma:T}\tlref{it3_lma:T} implies $dx \neq d$, and thus $\alpha(d)\aw x \neq \alpha(d)$.
In order to fulfill this inequality and condition~\eqref{eq1_clr:witness2}, $\alpha(d)$ must equal $\aw\ell\aw d$. 
It remains to determine $\aw\ell$. From $xd=d$ follows $\aw x \alpha(d) = \alpha(d)$. Thus $\alpha(d) = \aw x \aw d$. 
However $\alpha(yd) = \aw y\aw x\aw d = \aw x\aw d = \alpha(d)$.
Thus $\alpha$ is not injective, which yields a contradiction.
We proved that $T$ and $\aw T$ are not isomorphic.

Now the $(\Rightarrow)$ direction of the corollary is immediate from Lemma~\ref{lma:T}. 
The $(\Leftarrow)$ direction follows from the fact that none of the four bands described by Lemma~\ref{lma:T} satisfies \lam.
\end{proof}

\begin{lma}\label{lma:smpband_nphard_01}
Let $S$ be a finite band that does not satisfy \lam{} or \lamd{}. Then $\smp(S)$ is \np-hard.
\end{lma}

\begin{proof}
We may assume that $S$ does not satisfy \lam.
Let $d,e,x,y,h \in S$ witness $S\not\models\lam$ such that properties \lref{it1_lma:T} to~\lref{it3_lma:T} of Lemma~\ref{lma:T} hold. Denote $h$ by $1$ and let $T:=\subuni{d,e,x,y,1}$.
We reduce the Boolean satisfiability problem \sat{} to $\smp(T)$.
\sat{} is \np-complete \cite{Cook1971} and defined as follows.
\dproblem{\sat}{
clauses $C_1,\ldots, C_n \subseteq \{x_1,\ldots, x_k, \neg x_1,\ldots, \neg x_k\}$}
{Do truth values for $x_1,\ldots, x_k$ exist for which the Boolean formula 
$\phi( x_1,\ldots, x_k ) := (\bigvee C_1) \wedge\ldots\wedge (\bigvee C_n)$ 
is true?}
Fix a \sat{} instance $C_1,\ldots,C_n$ on $k$ variables.
For all $j\in\rg{k}$ we may assume that $x_j$ or $\neg x_j$ occurs in some clause $C_i$.
We define the corresponding $\smp(T)$ instance
\begin{align*}
A := \{u, v, a_1^0 ,\ldots, a_k^0, a_1^1,\ldots, a_k^1\} \subseteq T^{n+2k}, \quad b \in T^{n+2k}.
\end{align*}
The first $n$ positions of the tuples correspond to the $n$ clauses. The remaining $2k$ positions control the order in which tuples can be multiplied. Let
\begin{align*}
\arraycolsep=4pt
\def\arraystretch{1.25}
\begin{array}[c]{r@{}r@{}cccccccccc@{}l}
b     &{}:= (\acs&  de  &\cdots&  de  &  de  &\cdots&      &\cdots     &     &\cdots&  de &\acs), \\
u     &{}:= (\acs&  d   &\cdots&  d   &  d   &\cdots&      &\cdots     &     &\cdots&   d &\acs), \\
v     &{}:= (\acs&  xe  &\cdots&  xe  &  y   &\cdots&      &\cdots     &     &\cdots&   y &\acs), \\
a_j^0 &{}:= (\acs&      &      &      & \qh  &\cdots& \qh  &x\acs\acs e& \qh &\cdots& \qh &\acs)\quad\text{for } j\in\rg{k}, \\
a_j^1 &{}:= (\acs&      &      &      & \qh  &\cdots& \qh  &e\acs\acs x& \qh &\cdots& \qh &\acs)\quad\text{for } j\in\rg{k}. \vspace{-8pt} \\
      &{}        &                                                        
\rlap{$\underbrace{\phantom{xe\acs\acs{\cdots}\acs\acs{}xe}\hspace{1.5pt}}_{n}$}\phantom{xe}  &&& 
\rlap{$\underbrace{\phantom{de\acs\acs{\cdots}\acs\acs{}1}\hspace{1.5pt}}_{2j-2}$}\phantom{de}  &&&
\underbrace{\phantom{x\acs\acs{}e}\hspace{1.5pt}}_{2}  &
\rlap{$\underbrace{\phantom{1\acs\acs{\cdots}\acs\acs{}de}\hspace{1.5pt}}_{2k-2j}$}\phantom{1}  &&& \\
\end{array}
\end{align*}
For $j\in\rg{k}$ and $i\in\rg{n}$ let
\begin{align}\label{eq10_lma:smpband_nphard_01}
a_j^0(i) &:= 
\begin{cases}
e & \text{if } \neg x_j \in C_i, \\
1 & \text{otherwise,}
\end{cases} \qquad
a_j^1(i) := 
\begin{cases}
e & \text{if } x_j \in C_i, \\
1 & \text{otherwise.}
\end{cases}
\end{align}
Note that the size of the \sat{} instance is at least linear in the number of clauses $n$ and in the number of variables $k$.
Hence we have a polynomial reduction from \sat{} to $\smp(T)$.
In the remainder of the proof we show that
\begin{align}\label{eq20_lma:smpband_nphard_01}
\text{the Boolean formula $\phi$ is satisfiable if and only if $b \in \subuni A$.}
\end{align}

For the $(\Rightarrow)$ direction let $z_1,\ldots, z_k \in \{0,1\}$ such that\ $\phi(z_1,\ldots, z_k) = 1$. We claim that 
\begin{align}\label{eq30_lma:smpband_nphard_01}
u a_1^{z_1} \cdots a_k^{z_k} v = b.
\end{align}
For $i \in \rg{n}$
the clause $\bigvee C_i$ is satisfied under the assignment $x_1\mapsto z_1,\ldots,x_k\mapsto z_k$.
Thus there is a $j\in\rg{k}$ such that
$x_j \in C_i$ and $z_j = 1$, or $\neg x_j \in C_i$ and $z_j = 0$. 
In both cases $a_j^{z_j}(i) = e$ by~\eqref{eq10_lma:smpband_nphard_01}.
Thus $a_1^{z_1} \cdots a_k^{z_k}(i) = e$, and hence
\[ u a_1^{z_1} \cdots a_k^{z_k} v(i) = dexe = de = b(i) \]
by the multiplication table in Lemma~\ref{lma:T}.
For $i\in\rg{2k}$ we have $a_1^{z_1}\cdots a_k^{z_k}(n+i) \in \{x,e\}$. Thus 
\[ u a_1^{z_1} \cdots a_k^{z_k} v(n+i) \in \{dxy,dey\} = \{de\}. \] 
We proved \eqref{eq30_lma:smpband_nphard_01}. Thus $b\in\subuni{A}$.

For the $(\Leftarrow)$ direction of~\eqref{eq20_lma:smpband_nphard_01} assume $b \in \subuni{A}$. 
It is easy to see that $b = ubv$. Thus there is a minimal $\ell \in \mathbb{N}_0$ such that $b = u g_1 \cdots g_\ell v$ for some $g_1,\ldots,g_\ell\in A$.
We claim that 
\begin{equation}\label{eq60:lma:smpband_nphard_01}
u,v \not\in \{g_1, \cdots, g_\ell\}. 
\end{equation} If $g_j = u$ for some $j \in \rg{\ell}$, 
then $u g_1 \cdots g_j = u$ by Lemma~\ref{lma:band_rule1}\tlref{it1_lma:band_rule1}. Thus $b = u g_{j+1} \cdots g_\ell v$, contradicting the minimality of $\ell$. By a similar argument $v \not\in \{g_1,\ldots,g_{\ell}\}$. We proved~\eqref{eq60:lma:smpband_nphard_01}. Thus 
\[b = u a_{j_1}^{z_1} \cdots a_{j_\ell}^{z_\ell} v\] 
for some $j_1,\ldots,j_\ell \in \rg{k}$ and $z_1,\ldots,z_\ell \in \{0,1\}$. For $r,s \in \rg{\ell}$ we claim:
\begin{equation}\label{eq70:lma:smpband_nphard_01}
\text{If $j_r=j_s$, then $z_r=z_s$.}
\end{equation} 
Suppose there is a minimal index $r\in\rg{\ell}$ such that $j_r=j_s$ and $z_r\neq z_s$ for some $s\in\{r+1,\ldots,\ell\}$.
Then there is an $i \in \{2j_r-1,2j_r\}$ such that $a_{j_r}^{z_r}(n+i)=x$ and $a_{j_s}^{z_s}(n+i)=e$. 
By the minimality of $r$ we have $a_{j_1}^{z_1}(n+i)=\ldots=a_{j_{r-1}}^{z_{r-1}}(n+i)=1$.
Thus $a_{j_1}^{z_1} \cdots a_{j_s}^{z_s}(n+i) = xe$, and hence  $a_{j_1}^{z_1} \cdots a_{j_\ell}^{z_\ell}(n+i) = xe$. Therefore
\begin{align*}
u a_{j_1}^{z_1} \cdots a_{j_\ell}^{z_\ell} v(n+i) = dxey = dxe \ne b(n+i),
\end{align*}
which contradicts our assumption. We proved~\eqref{eq70:lma:smpband_nphard_01}. 

Now we define an assignment 
\begin{align*}
\theta \colon x_{j_1}&\mapsto z_1,\ldots,x_{j_\ell}\mapsto z_\ell, \\
x_j &\mapsto 0 \quad\text{for } j \in \rg k \setminus \{ j_1,\ldots, j_\ell \},
\end{align*}
and show that 
\begin{equation}\label{eq90:lma:smpband_nphard_01}
\text{$\theta$ satisfies the formula $\phi$.}
\end{equation}

Let $i \in \rg{n}$. We show that $\theta$ satisfies $\bigvee C_i$. Observe that $a_{j_1}^{z_1} \cdots a_{j_\ell}^{z_\ell}(i)$ is either $1$ or $e$. In the first case $u a_{j_1}^{z_1} \cdots a_{j_k}^{z_k} v(i) = dxe \neq b(i)$, which is a contradiction. Thus $a_{j_1}^{z_1} \cdots a_{j_\ell}^{z_\ell}(i) = e$. 
Since not all factors in $a_{j_1}^{z_1} \cdots a_{j_k}^{z_k}(i)$ can be $1$, we have $a_{j_r}^{z_r}(i)=e$ for some $r\in\rg{\ell}$. 
By~\eqref{eq10_lma:smpband_nphard_01} either $z_r=1$ and $x_{j_r} \in C_i$, or $z_r=0$ and $\neg x_{j_r} \in C_i$. In both cases $\theta$ satisfies $\bigvee C_i$. We proved~\eqref{eq90:lma:smpband_nphard_01} and \eqref{eq20_lma:smpband_nphard_01}.
\end{proof}

Finally we state an alternative version of our dichotomy result Theorem~\ref{thm:dichotomy_bands_intro}.

\begin{thm}\label{thm:dichotomy_bands}
Let $S$ be a finite band. Then $\smp(S)$ is in \ptime{} if one of the following equivalent conditions holds:
\begin{enumerate}
\item\label{it10_thm:dichotomy_bands} 
$S$ satisfies \lam{} and \lamd{}.
\item\label{it20_thm:dichotomy_bands} 
None of the four bands given in Lemma~\ref{lma:T} embeds into $S$ or $\dsgp S$.
\end{enumerate}
Otherwise $\smp(S)$ is \np-complete.
\end{thm}

\begin{proof}
The conditions~\lref{it10_thm:dichotomy_bands} and~\lref{it20_thm:dichotomy_bands} are equivalent by Corollary~\ref{clr:witness2}. If they are fulfilled, then $\smp(S)$ is in \ptime{} by Theorem~\ref{thm:smp_in_p}. Otherwise $\smp(S)$ is \np-hard by Lemma~\ref{lma:smpband_nphard_01} and in \np{} by Theorem~\ref{thm:smp_band_in_np}.
\end{proof}

\begin{proof}[Proof of Theorem~\ref{thm:dichotomy_bands_intro}]
Immediate from Theorem~\ref{thm:dichotomy_bands}.
\end{proof}

\section{Proof of Theorems~\ref{thm:genvar_diffcomp} and~\ref{thm:greatest_var_smp_tract}}
\label{sec:genvar_diffcomp}

In this section we prove the last two theorems of the introduction.

\begin{dfn}\label{dfn:s9s10}
Let $S_9$ and $S_{10}$ be the bands with the following multiplication tables.
\begin{align*}
\text{\footnotesize{$
\begin{array}{r|rrrrrrrrr}
S_9 & 1 & 2 & 3 & 4 & 5 & 6 & 7 & 8 & 9 \\
\hline
 1 & 1 & 2 & 3 & 4 & 5 & 6 & 7 & 8 & 9 \\
 2 & 2 & 2 & 4 & 4 & 5 & 6 & 7 & 8 & 9 \\
 3 & 3 & 3 & 3 & 3 & 3 & 6 & 7 & 8 & 9 \\
 4 & 4 & 4 & 4 & 4 & 4 & 6 & 7 & 8 & 9 \\
 5 & 5 & 5 & 5 & 5 & 5 & 6 & 7 & 8 & 9 \\
 6 & 6 & 7 & 8 & 9 & 8 & 6 & 7 & 8 & 9 \\
 7 & 7 & 7 & 9 & 9 & 8 & 6 & 7 & 8 & 9 \\
 8 & 8 & 8 & 8 & 8 & 8 & 6 & 7 & 8 & 9 \\
 9 & 9 & 9 & 9 & 9 & 9 & 6 & 7 & 8 & 9 \\
\end{array}$}} \qquad
\text{\footnotesize{$
\begin{array}{r|rrrrrrrrrr}
S_{10}&  1 &  2 &  3 &  4 &  5 &  6 &  7 &  8 &  9 & 10 \\ 
\hline
    1 &   1 &  2 &  3 &  4 &  5 &  6 &  7 &  8 &  9 & 10 \\
    2 &   2 &  2 &  4 &  4 &  5 &  6 &  7 &  8 &  9 & 10 \\
    3 &   3 &  3 &  3 &  3 &  3 &  6 &  7 &  8 &  9 & 10 \\
    4 &   4 &  4 &  4 &  4 &  4 &  6 &  7 &  8 &  9 & 10 \\
    5 &   5 &  5 &  5 &  5 &  5 &  6 &  7 &  8 &  9 & 10 \\
    6 &   6 &  7 &  8 &  9 & 10 &  6 &  7 &  8 &  9 & 10 \\
    7 &   7 &  7 &  9 &  9 & 10 &  6 &  7 &  8 &  9 & 10 \\
    8 &   8 &  8 &  8 &  8 &  8 &  6 &  7 &  8 &  9 & 10 \\
    9 &   9 &  9 &  9 &  9 &  9 &  6 &  7 &  8 &  9 & 10 \\
   10 &  10 & 10 & 10 & 10 & 10 &  6 &  7 &  8 &  9 & 10 \\
\end{array}$}}
\end{align*}
\end{dfn}

\noindent
Note that $S_9$ is isomorphic to the $9$-element band from Lemma~\ref{lma:T}\tlref{it4_lma:T}\lref{it4d_lma:T}
by renaming the elements as follows.
\newcommand{\sw}[1]{\makebox[\widthof{$dxe$\ }]{#1}} 
\[ \arraycolsep=0pt
\begin{array}{ccccccccc}
 h &  x &  e &  xe &  y &  d &  dx &  de &  dxe \\
\sw 1 & \sw 2 & \sw 3 &  \sw 4 & \sw 5 & \sw 6 &  \sw 7 &  \sw 8 &  \sw 9 \\
\end{array} \]

For the next result recall $G_n$, $H_n$, and $I_n$ from Definition~\ref{dfn:bands_Gn_Hn_In}.

\begin{lma}\label{lma:S9S10_var}
The bands $S_9$ and $S_{10}$ both generate the variety ${[\dwrd G_3 \approx \dwrd I_3]}$.
Furthermore $S_9$ is the homomorphic image of $S_{10}$ under
\[ \alpha \colon S_{10} \to S_9,\ x \mapsto \begin{cases} 
x & \text{if } x\le9, \\
8 & \text{if } x= 10.
\end{cases}\]
\end{lma}
\begin{proof}
From the multiplication tables it is immediate that $\alpha$ is a homomorphism.
Using the software \gap{} \cite{GAP4,GAP_semigroups} it is easy to show that both $S_9$ and $S_{10}$ satisfy the identity ${\dwrd G_3 \approx \dwrd I_3}$. 
It remains to show that $S_9$ and $S_{10}$ do not belong to a proper subvariety of ${[\dwrd G_3 \approx \dwrd I_3]}$.
By Figure~\ref{fig:varieties_of_bands} every proper subvariety of $[\dwrd G_3 \approx \dwrd I_3]$ is contained in ${[G_4 \approx H_4]}$.
For $v:=(2,1,3,6)$ and $S\in\{S_9,S_{10}\}$ we have 
\[ G_4^{S}(v) = 6123 = 9, \quad H_4^{S}(v) = 6123613123 = 8.\]
Thus neither $S_9$ nor $S_{10}$ satisfies $G_4 \approx H_4$. 
\end{proof}

\begin{lma}\label{lma:S10lambda_S9not}
The band $S_{10}$ satisfies $\lam$ and $\lamd$, whereas $S_9$ does not satisfy $\lam$.
\end{lma}
\begin{proof}
By Corollary~\ref{clr:witness2} $S_9$ does not satisfy \lam{}. 
From the multiplication table for $S_{10}$ we see that every $9$-element subsemigroup is of the form
$S_{10}\setminus\{s\}$ for $s\in\{1,2,3,5,6\}$. None of these semigroups is isomorphic to $S_9$ or its dual $\dsgp S_{9}$. Thus $S_{10}$ satisfies $\lam$ and $\lamd$ by Corollary~\ref{clr:witness2}.
\end{proof}

\begin{proof}[Proof of Theorem~\ref{thm:genvar_diffcomp}]
Items \lref{it1_thm:genvar_diffcomp} and~\lref{it2_thm:genvar_diffcomp} follow from Lemma~\ref{lma:S9S10_var}, and item~\lref{it3_thm:genvar_diffcomp} from Theorem~\ref{thm:dichotomy_bands} and Lemma~\ref{lma:S10lambda_S9not}.
\end{proof}

\begin{lma}\label{lma:if_g4h4_then_lam}
Let $S$ be a finite band that satisfies $G_4\approx H_4$. Then $S$ satisfies \lam.
\end{lma}
\begin{proof}
\renewcommand{\qh}{h}
Let $d,e,x,y,\qh \in S$ such that 
\begin{align*}
\arraycolsep=0pt
\begin{array}{c}
\begin{array}{rl}
\qd \qx \qy \qe {}&{} = \qd \qe \\
\qh \qx {}&{} = \qx \\
\qh \qe {}&{} = \qe \end{array}\\
\qd \gjle \qe \gjle \qx, \qy
\end{array}
\end{align*}
Set $v := (y,x,e,d)$.
From $S\models G_4 \approx H_4$ and the definitions of $G_4$ and $H_4$ follows
\begin{align*} 
dxye = G_4^S(v) = H_4^S(v) = dxyedxexye. 
\end{align*}
Lemma~\ref{lma:band_rule1}\tlref{it1_lma:band_rule1} implies $dxyed=d$ and $exye=e$.
Therefore the right hand side equals $dxe$, and thus $de=dxye=dxe$. 
Hence $S$ satisfies $\lam$.
\end{proof}

\begin{proof}[Proof of Theorem~\ref{thm:greatest_var_smp_tract}]
From Figure~\ref{fig:varieties_of_bands} we see that
\[ [\dwrd G_4 G_4\approx \dwrd H_4 H_4] = [G_4\approx H_4]\cap[\dwrd G_4 \approx \dwrd H_4]. \] 
Let $S$ be a finite band in this variety.
Since $S$ and $\dsgp S$ belong to $[G_4\approx H_4]$, both bands satisfy \lam{} by Theorem~\ref{lma:if_g4h4_then_lam}. Thus $S$ satisfies \lam{} and \lamd. By Theorem~\ref{thm:dichotomy_bands} $\smp(S)$ is in \ptime.

Now let $\mathcal V$ be a variety of bands greater than $[\dwrd G_4 G_4\approx \dwrd H_4 H_4]$.
First assume $\mathcal V \not\subseteq {[G_4\approx H_4]}$. Then ${[\dwrd G_3\approx \dwrd I_3]} \subseteq \mathcal V$ by Figure~\ref{fig:varieties_of_bands}. By Lemma~\ref{lma:S9S10_var} $S_9$ belongs to $\mathcal V$.
If $\mathcal V \not\subseteq [\dwrd G_4\approx \dwrd H_4]$, then the dual band $\dsgp S_9$ belongs to $\mathcal V$ by a similar argument. 
Since the \smp{} for both $S_9$ and $\dsgp S_9$ is \np-complete by Lemma~\ref{lma:T}, the result follows.
\end{proof}

\section{Conclusion and open problems}

In ~\cite{BMS2015} it was shown that the \smp{} for every finite semigroup is in \pspace{}.
Moreover, semigroups with \np{}-complete and \pspace{}-complete \smp{}s were provided.
In the present paper we presented the first examples of completely regular semigroups with \np-hard \smp{}. We  established a \ptime/\np-complete dichotomy for the case of bands.
However, it is unknown if the \smp{} for every completely regular semigroup is in \np{}.
Thus the following problem is open.

\begin{prb}
Is the \smp{} for every finite completely regular semigroup in \np{}?
In particular, is there a \ptime/\np-complete dichotomy for completely regular semigroups?
\end{prb}

So far there is no algebra known for which the \smp{} is neither in \ptime{} nor complete in \np{}, \pspace{}, or \exptime{}.
In the case of semigroups, the following open problem arises.

\begin{prb}
Is there a \ptime/\np-complete/\pspace-complete trichotomy for the \smp{} for semigroups?
\end{prb}

\bibliographystyle{abbrv} 

\end{document}